\documentclass[11pt]{amsart}

\usepackage[bookmarks=true]{hyperref}

\usepackage{mathptmx}
\usepackage{amsmath}
\usepackage{amssymb}
\usepackage{amsthm}
\usepackage{array}
\usepackage[all,tips]{xy}
\usepackage{enumerate}
\usepackage{graphicx}
\usepackage{lmodern}
\usepackage{mathrsfs}

\newtheorem{theorem}{Theorem}
\newtheorem{lemma}[theorem]{Lemma}

\newtheorem{cor}[theorem]{Corollary}

\theoremstyle{definition}

\DeclareMathOperator{\SL}{SL}

\DeclareMathOperator{\PU}{PU}
\DeclareMathOperator{\PSO}{PSO}
\DeclareMathOperator{\PSL}{PSL}

\DeclareMathOperator{\Aut}{Aut}

\newcommand{\N}{\mathbb N}

\newcommand{\R}{\mathbb R}
\newcommand{\C}{\mathbb C}

\newcommand{\f}{\mathfrak f}

\newcommand{\M}{\overline M}
\newcommand{\Rad}{R}

\newcommand{\bL}{\overline{L}}

\newcommand{\Hy}{\mathcal H}

\newcommand{\Af}{\mathbb{A}_\mathrm{f}}

\newcommand{\V}{V}
\newcommand{\Vf}{\V_{\mathrm{f}}}

\newcommand{\Snc}{\mathscr{S}}
\newcommand{\piS}{\pi_\Snc}

\newcommand{\nqs}{T}
\newcommand{\Spl}{S}
\newcommand{\Ram}{R}

\newcommand{\cG}{c_\G}

\newcommand{\PA}{\Lambda}

\renewcommand{\P}{P}
\newcommand{\Po}{P^{(1)}}
\newcommand{\Poo}{P^{(2)}}
\newcommand{\Pk}{K}

\newcommand{\K}{\mathcal{K}}

\newcommand{\G}{\mathrm G}
\newcommand{\bG}{\overline{\G}}

\newcommand{\GG}{\mathcal G}

\newcommand{\gC}{\mathfrak{g}_\C}

\newcommand{\X}{X}

\newcommand{\zk}{\mathcal{O}}

\newcommand{\bs}{\backslash}

\newcommand{\CG}{\mathrm C}

\newcommand{\tA}{\mathrm{A}}

\newcommand{\tD}{\mathrm{D}}

\begin{document}
\title{Arbitrarily large families of spaces of the same volume}

\author{Vincent Emery}
\thanks{Supported by Swiss National Science Foundation, Project number {\tt PP00P2-128309/1}
}
\address{
Section de math\'ematiques\\
Universit\'e de Gen\`eve\\
2--4 rue du Li\`evre, Case postale~64\\
CH-1211 Gen\`eve~4\\
SWITZERLAND
}
\email{vincent.emery@gmail.com}

\date{\today}


\begin{abstract}

	In any connected non-compact semi-simple Lie group without
	factors locally isomorphic to $\SL_2(\R)$, there can be only
	finitely many lattices (up to isomorphism) of a given covolume.
	We show that there exist arbitrarily large families of pairwise
	non-isomorphic arithmetic lattices of the same covolume. We
	construct these lattices with the help of Bruhat-Tits theory,
	using Prasad's volume formula to control their covolumes.

\end{abstract}

\maketitle

\section{Introduction}
\label{sec:intro}

Let $\GG$ be a connected  semi-simple real Lie group without compact
factors. For simplicity we will suppose that $\GG$ is adjoint (i.e., with trivial
center), though this is not a major restriction in this article.  Any
choice of a Haar measure $\mu$ on $\GG$ assigns a \emph{covolume}
$\mu(\Gamma \bs \GG) \in \R_{> 0}$ to each lattice $\Gamma$ in $\GG$.
Wang's theorem~\cite{Wang72} asserts that there exist only finitely many
irreducible lattices (up to conjugation) of bounded covolumes in $\GG$
unless $\GG$ is isomorphic to $\PSL_2(\R)$ or $\PSL_2(\C)$.  In
particular, there exist only finitely many irreducible lattices in $\GG$
of a given covolume.  For $\GG$ isomorphic to $\PSL_2(\C)$ this property
is still true, as follows from the work of Thurston and J\o
rgensen~\cite[Ch.~6]{Thur80}.  In this paper we prove that the number of
lattices in $\GG$ of the same covolume can be arbitrarily large.  In most
cases, arbitrarily large families of lattices of equal covolume appear in
the commensurability class of any arithmetic lattice of $\GG$.  This is
the content of the following theorem. The symbol $\gC$ denotes the
complexification of the Lie algebra of $\GG$.

\begin{theorem}
	\label{thm:general}

	Let $\GG$ be a connected adjoint semi-simple real Lie group
	without compact factors.
	We suppose that $\gC$ has a simple factor that is not of type
	$\tA_1$, $\tA_2$ or $\tA_3$.
	Let $\Gamma$ be an  arithmetic lattice in $\GG$.
	Then, for every $m \in \N$, there exist a family of  $m$
	lattices commensurable to $\Gamma$ that are pairwise non-isomorphic
	and have the same covolume in $\GG$.
	These lattices can be chosen torsion-free.

\end{theorem}

Every arithmetic lattice $\Gamma \subset \GG$ is constructed with the
help of some algebraic group $\G$ defined over a number field $k$ (see
Section~\ref{ss:arithmetic-subgroup}). 
To prove Theorem~\ref{thm:general}, we use Bruhat-Tits theory to
construct families of arithmetic subgroups in $\G(k)$ that are
non-conjugate, and have equal covolume.
By strong (Mostow) rigidity one obtains the analogous result with ``pairwise
non-conjugate'' replaced with ``pairwise non-isomorphic''.
To control the covolume we
use some computations that appear in Prasad's volume formula~\cite{Pra89}. 
To ensure that the subgroups
constructed are not conjugate we need to exhibit parahoric subgroups in
$\G(k_v)$ (where $k_v$ is a non-archimedean completion of $k$) that are not
conjugate but of the same volume. This can be easily achieved 
when $\G$ is not of type $\tA_n$ and is split over $k_v$.
When $\G$ is of type $\tA_n$ the Bruhat-Tits building of a split $\G(k_v)$ has
more symmetries, and the argument must be slightly adapted. In
particular, there we need the assumption $n \ge 4$, which explains the excluded cases
in the statement of Theorem~\ref{thm:general}.
The simple Lie groups excluded are listed in Table~\ref{tab:exceptions}.
\begin{table}[t]
	\centering
	\begin{tabular}{ll}
	Type $\tA_1$: & $\PSL_2(\R)$  and  $\PSL_2(\C)$; \\[3pt]
	Type $\tA_2$: & $\PSL_3(\R)$, $\PSL_3(\C)$  and $\PU(2,1)$; \\[3pt]
	Type $\tA_3$: & $\PSL_4(\R)$, $\PSL_4(\C)$, $\PSO(5,1)$, $\PU(3,1)$ and 
	$\PU(2,2)$.
	\end{tabular}
	\medskip
	\medskip
	\caption{Simple Lie groups not covered in
	Theorem~\ref{thm:general}}
	\label{tab:exceptions}
\end{table}
	\label{eq:exceptions}

For the Lie groups of type $\tA_2$ and $\tA_3$ we can use algebraic
groups that are outer forms (type $^2\tA_2$ and $^2\tA_3$) to show the
existence of arbitrarily large families of arithmetic lattices  of the same
covolume. In contrast with Theorem~\ref{thm:general}, now each family
corresponds to a different commensurability class.

\begin{theorem}
	\label{thm:exceptions}

	Let $\GG$ be a connected adjoint semi-simple Lie group without compact
	factors. We suppose that $\gC$ contains only factors of type
	$\tA_2$ (resp. only factors of type $\tA_3$). Let $m \in \N$.
	Then there exists
	a family $\left\{ \Gamma_1, \dots, \Gamma_m \right\}$ of
	irreducible arithmetic lattices in $\GG$ such that for $i,j \in
	\left\{ 1,\dots m \right\}$:
	\begin{enumerate}
		\item $\Gamma_i$ is commensurable to $\Gamma_j$;
		\item $\Gamma_i$ and $\Gamma_j$ have the same covolume in
			$\GG$;
		\item if $i \neq j$, then $\Gamma_i$ and $\Gamma_j$ are
			not isomorphic.
	\end{enumerate}
	The lattices $\left\{  \Gamma_i
	\right\}$ can be  chosen torsion-free.  Moreover, they  can be
	chosen cocompact. They can be chosen non-cocompact unless
	there are no such lattices in $\GG$.

\end{theorem}

It follows from Margulis' arithmeticity theorem that irreducible
lattices can only exist in a Lie group $\GG$ that is isotypic (i.e., for
which all the simple factors of $\gC$ have the same type), so that the
assumptions in Theorem~\ref{thm:exceptions} are minimal.  The existence
of irreducible cocompact lattices in any isotypic $\GG$ was proved by
Borel and Harder~\cite{BorHar78}.  Non-compact
irreducible quotients of $\GG$ do not always exist. For example there is no such
quotient of $\PU(3,1) \times \PSO(5,1)$ (this example is detailed
in~\cite[Prop. (15.31)]{WittMorr08}). A general criterion for the
existence of non-cocompact arithmetic lattices appears in the work of
Prasad-Rapinchuk~\cite{PraRap06}, where the authors extend the results
of~\cite{BorHar78}. The proof of
Theorem~\ref{thm:exceptions} uses these existence results.

By Wang's theorem, it is clear that the covolume common to the lattices
of a family grows with the size of the family. Even though in this
article we focus on qualitative results, we note that the proofs of
Theorems~\ref{thm:general}--\ref{thm:exceptions} could be used to obtain
quantitative results on the growth of the covolume with the size of the
family.

\medskip

We now discuss  the geometric significance of our results. Let $\X$ be
the symmetric space associated with  $\GG$, that is $\X =
\GG/K$ for a maximal compact subgroup $K \subset \GG$. This class of
spaces includes the \emph{hyperbolic $n$-space} $\Hy^n$; we have that
$\Hy^2$ is associated with $\GG = \PSL_2(\R)$, and $\Hy^3$ with $\GG =
\PSL_2(\C)$.  For a torsion-free irreducible lattice $\Gamma \subset
\GG$, the locally symmetric space $\Gamma \bs \X$ will be called
\emph{an $\X$-manifold} (in particular it is irreducible and has finite
volume).  The following result follows directly from
Theorems~\ref{thm:general}--\ref{thm:exceptions} and the existence of
cocompact arithmetic lattices in $\GG$ (see for
instance~\cite[Theorem~1]{PraRap06}).

\begin{cor}
	\label{cor:same-volume}

	Let $X$ be a Riemannian symmetric space of non-compact type that contains
	no factor isometric to $\Hy^2$ or $\Hy^3$, and suppose that
	irreducible quotients of $X$ do exist.
	Then there exist arbitrarily large families of pairwise
	non-isometric commensurable compact $X$-manifolds having the same volume. 
	The analogue statement with non-compact $X$-manifolds is true unless
	all $X$-manifolds are compact.
\end{cor}

The result for $X = \Hy^3$ was proved by Wielenberg  for the case of
non-compact manifolds~\cite{Wiel81}, and later by Apasanov-Gutsul for
compact manifolds~\cite{ApaGut92}. For $X = \Hy^4$ the result with
non-compact manifolds was proved by Ivan\v{s}i\'c in his
thesis~\cite{Ivan99}. All these results are obtained by geometric
methods. In~\cite{Zimm94} Zimmerman gave a new proof for $X = \Hy^3$ by
exhibiting examples of $\Hy^3$-manifolds $M$ with first Betti number
$\beta_1$ at least $2$, and showing that this property implies the
existence of arbitrarily large families of covering spaces of $M$ of
same degree. In~\cite{Lubo96} Lubotzky showed  that there exist (many)
hyperbolic manifolds with $\beta_1 \ge 2$ in every dimension. Thus for
all $X = \Hy^n$ we have a proof of Corollary~\ref{cor:same-volume} by
Zimmerman's method. Since  super-rigidity implies
that $H^1(\Gamma \bs X,\R) = 0$ for irreducible lattices $\Gamma$ in
$\GG$ with $\R\mathrm{-rank}(\GG) \ge 2$, the same approach cannot be
used  to prove the result in this situation. Conversely, it does not
seem that our method can be adapted to include the case of $\Hy^2$ and
$\Hy^3$.

Very recently, Aka constructed non-isomorphic arithmetic lattices that
have isomorphic profinite completions~\cite{Aka}. In particular, 
his construction gives arbitrarily large families of
lattices of equal covolume in the Lie group
$\SL_n(\C)$, for any $n \ge 3$.

\subsection*{Acknowledgements}
I would like to thank Misha Belolipetsky, Pierre de la Harpe, Gopal
Prasad and Matthew Stover for helpful
comments on an early version of this paper. I also thank Menny Aka for a
helpful correspondence on his preprint.

\section{Arithmetic lattices}
\label{sec:arithmetic-lattices}


We can obviously reduce the proof of Theorem~\ref{thm:general} to the
case of an irreducible $\Gamma$. Then, like in
Theorem~\ref{thm:exceptions},  $\GG$ is supposed to be isotypic.

\subsection{}
\label{ss:arithmetic-subgroup}

For generalities on arithmetic groups we refer the reader
to~\cite{Zimmer84} and~\cite{PlaRap94}.  We briefly explain here how
irreducible arithmetic lattices in $\GG$ are obtained.  Let $k$ be a
number field with ring of integers $\zk$. Let $\G$ be an absolutely
simple simply connected algebraic group defined over $k$.  We denote by
$\bG$ the adjoint group of $\G$, i.e., the $k$-group defined as $\G$
modulo its center, and by $\pi : \G \to \bG$ the natural isogeny. 
Let $\Snc$ be the set of archimedean  places $v$ of $k$ such that $\G(k_v)$ is non-compact. 
We denote by $\G_\Snc$
the product $\prod_{v \in \Snc} \G(k_v)$, and similarly for $\bG_\Snc$.
Note that $\G_\Snc$ is connected.
For any matrix realization of $\G$, the group $\G(\zk)$ is an irreducible
lattice in $\G_\Snc$. Suppose that the connected component
$(\bG_\Snc)^\circ$ of ${\bG_\Snc}$ is
isomorphic to $\GG$. Then $\pi$ extends  to a surjective map $\piS:
\G_\Snc \to \GG$. An irreducible lattice in $\GG$ is called \emph{an arithmetic lattice} if
it is commensurable with a subgroup of the form $\piS(\G(\zk))$ for some
$k$-group $\G$ as above.

In the following $\G$ will always be a $k$-group as above, which determines a
commensurability class of arithmetic lattices in $\GG$.

\subsection{}
\label{ss:adelic-group}

We denote  by $\Vf$  the set of finite places  of $k$, and by $\Af$ the ring of
finite ad\`eles of $k$.  For each $v \in \Vf$ we consider $k_v$ the
completion  of $k$ with respect to $v$, and $\zk_v \subset k_v$ its
associated valuation ring.  A collection $\P =
(\P_v)_{v \in \Vf}$ of compact subgroups $\P_v \subset \G(k_v)$ is
called \emph{coherent} if the product $\K_\P = \prod_{v \in \Vf} \P_v$
is open in the adelic group $\G(\Af)$ (see~\cite[Ch.~6]{PlaRap94} for
information on adelic groups). For example, for any matrix
realization of $\G$, the collection $(\G(\zk_v))_{v \in \Vf}$ is
coherent.  For a coherent collection $\P = (\P_v)$, the group
\begin{eqnarray}
	\PA_\P &=&  \G(k) \cap \prod_{v \in \Vf} \P_v,
	\label{eq:PA_P}
\end{eqnarray}
where $\G(k)$ is seen diagonally embedded into $\G(\Af)$, is an
arithmetic subgroup of $\G(k)$ (and thus an arithmetic lattice in
$\G_\Snc$). This follows from the equality $\G(\zk) = \G(k) \cap \prod_v
\G(\zk_v)$ together with the inequality
\begin{eqnarray}
	[\PA_\P : \PA_{\P'}] &\le& [\K_\P : \K_{\P'}],
	\label{eq:inequality-global-local}
\end{eqnarray}
valid for any two coherent collections $\P$ and $\P'$ with
$\P'_v \subset \P_v$ for each $v\in \Vf$. Since $\G$ is simply
connected, strong approximation holds~\cite[Theorem~7.12]{PlaRap94} and
it follows that~\eqref{eq:inequality-global-local} is in fact an
equality. We put this (known) result in the following lemma.

\begin{lemma}
	\label{lem:index-by-strong-approx}

	Let $\P = (\P_v)_{v \in \Vf}$ and $\P' = (\P'_v)_{v \in \Vf}$ be
	two coherent collections of compact subgroups such that $\P'_v
	\subset \P_v \subset \G(k_v)$ for all $v \in \Vf$. Then

	\begin{eqnarray*}
		[\PA_\P : \PA_{\P'}] &=& \prod_{v \in \Vf} [\P_v:\P_v'].
		\label{eq:index-by-strong-approx}
	\end{eqnarray*}

\end{lemma}

\subsection{}
\label{ss:non-conjugated}

For every field extension $L|k$ with algebraic closure $\bL$,
the group of $L$-points given by $\bG(L)$ is identified with the
inner automorphisms of $\G$ that are defined over $L$. Note that in
general $\bG(L)$ is larger than the image of $\G(L)$ in $\bG(\bL)$.

\begin{lemma}
	\label{lem:non-conjugated}

	Let $\P$ and $\P'$ be two coherent collections of compact
	subgroups $\P_v, \P_v' \subset \G(k_v)$. Suppose that there
	exist a place $w \in \Vf$ such that $\P_{w}$ and $\P_{w}'$
	are not conjugate by the action of $\bG(k_{w})$. Moreover, we
	suppose that $\P_{w}$ and $\P_{w}'$ contain the center of
	$\G(k_{w})$. Then
	$\piS(\PA_\P)$ and $\piS(\PA_{\P'})$ are not conjugate in~$\GG$.

\end{lemma}

\begin{proof}
	Let $\CG$ be the center of $\G$. 
	We may assume that each $\P_v$ (resp. $\P'_v$) contains the
	center $\CG(k_v)$. If not replace $\P_v$ by $\CG(k_v) \cdot
	\P_v$; the image $\piS(\PA_\P)$ does not change with this
	modification, and the hypothesis at $w$ is kept.

	Suppose that $\piS(\PA_\P)$ and $\piS(\PA_{\P'})$ are conjugate in $\GG$. 
	Then $\PA_{\P}$ and
	$\PA_{\P'}$ are conjugate under the action of $\GG \cong
	(\bG_\Snc)^\circ$. 
	Since arithmetic subgroups of $\G$ are
	Zariski-dense, we have more precisely that $\PA_\P$ and
	$\PA_{\P'}$ are conjugate by an element  $g \in  \bG(k)$.
   By strong approximation the closure of $\PA_\P$ (resp.
	$\PA_{\P'}$) in $\G(k_w)$ is $\P_w$ (resp. $\P'_w$), and it
	follows that $g$ conjugates $\P_w$ and $\P'_w$. 	
\end{proof}

\section{Parahoric subgroups and volume}
\label{sec:parahorics}

In the following we assume that the reader has some knowledge of
Bruhat-Tits theory. All the facts we need can be found in 
Tits' survey~\cite{Tits79}. See~\cite[\S 3.4]{PlaRap94} for a more
elementary introduction.

\subsection{}
\label{ss:parahorics}

Let $v \in \Vf$.
A \emph{parahoric subgroup} of $\G(k_v)$, a certain kind of compact open
subgroup of $\G(k_v)$, is by
definition the stabilizer of a simplex in the Bruhat-Tits building
attached to $\G(k_v)$. There are a finite number of conjugacy classes
of parahoric subgroups in $\G(k_v)$; these conjugacy classes in
$\G(k_v)$ correspond canonically to proper subsets of the local Dynkin diagram $\Delta_v$ of $\G(k_v)$. 
If $\P_v \subset \G(k_v)$ is a parahoric subgroup, we denote by
$\tau(\P_v) \subset \Delta_v$ its associated subset, and we call
it the \emph{type} of $\P_v$. Two parahoric subgroups $\P_v$ and $\P_v'$  can
be conjugate by an element of $\bG(k_v)$ only if there is an
automorphism of $\Delta_v$ that sends $\tau(\P_v)$ to $\tau(\P_v')$.  

\subsection{}
\label{ss:residual-group-scheme}

Let us denote by $\f_v$ the residual field of $k_v$. To each parahoric
subgroup $\P_v \subset \G(k_v)$, a smooth affine group scheme over
$\zk_v$ is associated in a canonical way~\cite[\S 3.4.1]{Tits79}. By
reduction modulo $v$, this determines in turn an algebraic group over
$\f_v$. Its maximal reductive quotient is a $\f_v$-group that will be
denoted by the symbol $\M_v$. The structure of $\M_v$ can be determined
from $\tau(\P_v)$ and the local index of $\G(k_v)$ by the procedure
described in~\cite[\S 3.5]{Tits79}.

\subsection{}
\label{ss:lang-isogeny}

Let $(\M_v,\M_v)$ be the commutator group of $\M_v$, and let $\Rad(\M_v)$
be the radical of $\M_v$. Both are defined over $\f_v$, and we have
(see~\cite[8.1.6]{Spring98})
\begin{eqnarray*}
	\label{eq:M_v-as-product-almost-direct}
	\M_v &=& (\M_v,\M_v) \cdot \Rad(\M_v).	
\end{eqnarray*}
The radical $\Rad(\M_v)$ is a central torus in $\M_v$, whose intersection
with $(\M_v,\M_v)$ is finite~\cite[7.3.1]{Spring98}. It follows that the
product map
\begin{eqnarray*}
	\label{eq:M_v-product-map}
	(\M_v,\M_v) \times \Rad(\M_v) &\to& \M_v
\end{eqnarray*}
is an isogeny. By applying Lang's isogeny theorem~\cite[Prop.~6.3]{PlaRap94}, 
we obtain that the order of $\M_v(\f_v)$ is given by the following:
\begin{eqnarray}
	\label{order-of-M_v-as-product}
	|\M_v(\f_v)| &=& |(\M_v,\M_v)(\f_v)| \cdot |\Rad(\M_v)(\f_v)|.
\end{eqnarray}


\medskip

\begin{theorem}[Prasad]
	
	\label{thm:prasad-formula}

	Let $\mu$ be a Haar measure on $\G_\Snc$. Then there exists a
	constant $\cG$ (depending on the algebraic group $\G$) such that for any coherent collection  $\P$ of
	parahoric subgroups $\P_v \subset \G(k_v)$, we have 
	\begin{eqnarray*}
		\mu(\Lambda_\P\bs \G_\Snc) &=& \cG 
		\prod_{v \in \Vf} \frac{|\f_v|^{\left( t_v + \dim
		\M_v \right)/2 }}{|\M_v(\f_v)|},
	\end{eqnarray*}
	where for each $v \in \Vf$ the integer $t_v$ depends only on the
	$k_v$-structure of $\G$.

\end{theorem}

This theorem is a much weaker form of Prasad's volume
formula, given in~\cite[Theorem~3.7]{Pra89}. In fact, Prasad's result
explicitly gives the value of $\cG$ for a natural normalization of the
Haar measure $\mu$. Moreover, the integers $t_v$ are explicitly known.
Since we want to prove qualitative
results, we will not need more than the statement of
Theorem~\ref{thm:prasad-formula}.

\section{Proof of Theorem~\ref{thm:general} }
\label{sec:proof-1}

We now prove Theorem~\ref{thm:general}, assuming that
the group $\GG$ is isotypic. Let $\Gamma \subset \GG$
be an irreducible arithmetic lattice, with $\G$ and $\bG$ the
associated $k$-groups as in Section~\ref{ss:arithmetic-subgroup}. We
retain all notation introduced above.

\subsection{}
\label{ss:cebotarev}

The group $\G$ is quasi-split over $k_v$ for almost all places
$v$~\cite[Theorem~6.7]{PlaRap94}. Let us denote by $\nqs$ the set of the
places $v \in \Vf$ such $\G$ is not quasi-split over $k_v$.
Let $\ell|k$ be the smallest Galois
extension such that $\G$ is an inner form over $\ell$ (see for
instance~\cite[Ch.~17]{Spring98}, where this field is denoted by
$E_\tau$). If $v \not
\in \nqs$ is totally split in $\ell|k$, i.e., if $\ell \subset k_v$, then
$\G$ is split over $k_v$. It follows from the Chebotarev density theorem
that the set of places $v \not \in \nqs$ that are totally split in $\ell|k$ is
infinite. Let us denote this infinite subset of $\Vf$ by $\Spl$.

\subsection{}
\label{ss:good-parahorics-general}

Let $v \in \Spl$. The local Dynkin diagram $\Delta_v$ of $\G(k_v)$ can
be found in~\cite[\S 4.2]{Tits79}. Let $n$ be the absolute rank of
$\GG$ (and of $\G$). We suppose first that $\GG$ (and
consequently $\G$ as well) is not of absolute type $\tA_n$. Then there
exist two vertices $\alpha_1,\alpha_2 \in \Delta_v$ such that $\alpha_1$
is hyperspecial and $\alpha_2$ is not. Let $\Po_v$ (resp. $\Poo_v$) be a
parahoric subgroup in $\G(k_v)$ of type $\tau(\Po_v) = \left\{ 
\alpha_1 \right\}$ (resp.
$\tau(\Poo_v) = \left\{  \alpha_2 \right\}$). Then $\Po_v$ and $\Poo_v$ are not conjugate
by the action of $\bG(k_v)$ (see Section~\ref{ss:parahorics}). Note also
that these two groups, being parahoric subgroups, contain the center of
$\G(k_v)$. We consider the subgroup $\M_v$ associated with $\Po_v$
(resp. associated with $\Poo_v$). In both cases $i=1,2$ the radical
$\Rad(\M_v)$ is a split torus of rank $n-1$ and the semi-simple
part  $(\M_v,\M_v)$ is of type $\tA_1$.
From~\eqref{order-of-M_v-as-product} we see that the order of
$\M_v(\f_v)$ is the same for $\Po_v$ and $\Poo_v$.

If $\G$ is of type $\tA_n$ then $\Delta_v$ is a cycle of $n+1$ vertices,
all hyperspecial. The group $\bG(k_v)$ acts simply transitively by
rotations on $\Delta_v$. 
Let us choose a labelling $\alpha_0, \dots, \alpha_n$ of
the vertices that follows an orientation of $\Delta_v$. We now consider $\Po_v$ with
$\tau(\Po_v) = \left\{ \alpha_0,\alpha_2 \right\}$, and $\Poo_v$ with $\tau(\Poo_v) =
\left\{  \alpha_0, \alpha_3 \right\}$. If $n\ge 4$ then no rotation
of $\Delta_v$ sends $\tau(\Po_v)$ to $\tau(\Poo_v)$, so that $\Po_v$ and
$\Poo_v$ are not conjugate by $\bG(k_v)$. Moreover, we can check as
above that  the order of $\M_v$ is the same for $\Po_v$ and $\Poo_v$.

\label{ss:good-parahorics-A_n}

\subsection{}
\label{ss:construct-families}

We consider a coherent collection $\P$ of parahoric subgroups $\P_v \subset
\G(k_v)$. 
Let $m \in \N$ and choose a finite subset $\Spl_m \subset \Spl$  of
length $m$. For each $v \in \Spl_m$ we replace $P_v$ by either $\Po_v$
or $\Poo_v$, and consider the arithmetic subgroup in $\G(k)$ associated with this
modified coherent collection. Thus we obtain $2^m$ different arithmetic
subgroups in $\G(k)$, and by Lemma~\ref{lem:non-conjugated} their images
in $\GG$ are pairwise non-conjugate. But by
Theorem~\ref{thm:prasad-formula} they all have the same covolume.

To obtain families of torsion-free lattices we make the following
change. Let us choose two distinct places $v_1, v_2 \in \Spl \setminus
\Spl_m$, and for $i=1,2$ replace $\P_{v_i}$ by its subgroup $\Pk_i$
defined as the kernel of the reduction modulo $v_i$. We denote this
modified coherent collection by $\P'$. Let $p_i$ be the characteristic
of $\f_{v_i}$. Then $\Pk_i$ is a
pro-$p_i$-group~\cite[Lemma~3.8]{PlaRap94}, and since $p_1 \neq p_2$ we
have that $\Pk_1 \cap \Pk_2$ is torsion-free. Thus $\Lambda_{\P'}$ is
torsion-free. The above construction with the coherent
collection $\P'$ instead of $\P$ now gives non-conjugate lattices in
$\GG$ that are torsion-free.  Using
Lemma~\ref{lem:index-by-strong-approx} we see that these sublattices
also share the same covolume.

\subsection{}
\label{ss:strong-rigidity}

Let $\Aut(\GG)$ be the automorphism group of $\GG$. Then
$\Aut(\GG)/\GG$ (where $\GG$ acts on itself as inner automorphisms)
is a group whose order is bounded by the symmetries of the Dynkin
diagram of $\GG$. In particular, it is a finite group. By letting
$m$ tends to infinity, we have constructed arbitrarily large families of
non-conjugate lattices in $\GG$ of the same covolume. By considering each
family modulo the equivalence induced by the action of
$\Aut(\GG)/\GG$, we see that there exist arbitrarily large families of
lattices that are not conjugate by $\Aut(\GG)$. Since strong rigidity
holds for all the lattices under consideration (see~\cite[\S 5.1]{Zimmer84}
and the references given there), we get that these families consist of
non-isomorphic lattices.

\section{Proof of Theorem~\ref{thm:exceptions} }
\label{sec:proof-2}

We now give the proof of Theorem~\ref{thm:exceptions}. Thus we suppose
that $\gC$ has only factors of type $\tA_n$ (with $n = 2$ or $n = 3$). Let
$m \in \N$.

\subsection{}
\label{ss:good-extension-ell}

Let $k$ be a number field that has as many complex places as there are
simple factor of $\GG$ isomorphic to $\PSL_{n+1}(\C)$. Let $\ell|k$ be a
quadratic extension having one complex place for each factor of $\GG$
that is projective unitary (i.e., of the form $\PU(p,q)$) or isomorphic to
$\PSL_{n+1}(\C)$.  Using approximation for $k$ (see~\cite[Theorem~(3.4)]{Neuk99}) it is
possible to choose $\alpha \in k$ such that $\ell = k(\sqrt{\alpha})$
is as above with the additional property that for the set $\Ram \subset \Vf$ of
ramified places in $\ell|k$ we have $2^{\# \Ram} \ge m$.

\subsection{}
\label{ss:G-from-PraRap}

Let $\G_0$ be the quasi-split simply connected $k$-group of type $\tA_n$
with splitting field $\ell$. By~\cite[Theorem~1]{PraRap06}, there exists
an inner form $\G$ of $\G_0$ such that $\G|k_v$ is quasi-split for all
$v \in \Ram$ and such that $(\bG_\Snc)^\circ \cong \GG$. The group $\G$ can be chosen to be $k$-isotropic unless the
condition~(1) in~\cite{PraRap06} is not satisfied at infinite places, in
which case there is no isotropic $k$-group $\G$ with $(\bG_\Snc)^\circ
\cong \GG$. We can always choose $\G$ to be anisotropic, by specifying
in~\cite[Theorem~1]{PraRap06} that $\G$ is $k_v$-anisotropic at some $v
\in \Vf \setminus \Ram$. 

\subsection{}
\label{ss:end-pf-thm-2}

The local Dynkin diagram $\Delta_v$ of $\G(k_v)$ for $v \in \Ram$ is
shown in~\cite[\S 4.2]{Tits79}; it is named $\mbox{C--BC}_1$ for the
type $\tA_2$, and $\mbox{C--B}_2$ for $\tA_3$ ($=\tD_3$). With this
diagram at hand we can easily construct (similarly to
Section~\ref{ss:good-parahorics-general}) a pair of non-conjugate
parahoric subgroups of $\G(k_v)$ ($v \in \Ram$) that have equal volume.
Taking them as part of coherent collection we produce $m$ pairwise
non-conjugate arithmetic subgroups that, by
Theorem~\ref{thm:prasad-formula}, are of the same covolume in $\GG$.
By Godement's compactness criterion, these lattices are cocompact
exactly when $\G$ is anisotropic.  The last
steps of the proof are verified exactly as in
Sections~\ref{ss:construct-families}--\ref{ss:strong-rigidity}.

\bibliographystyle{amsplain}
\bibliography{/home/vincent/Math/mes-textes/emery-bib}

\end{document}